\documentclass[a4paper,12pt]{amsart}
\usepackage[T1]{fontenc}
\usepackage[centering,margin=2.3cm]{geometry}
\usepackage[latin1]{inputenc}
\usepackage{amsthm, amsmath}   
\usepackage{latexsym,amssymb}
\usepackage{algorithmic}
\usepackage{amssymb}
\usepackage{graphicx,subfigure}
\usepackage{times}
\usepackage{enumerate}
\usepackage[usenames,dvipsnames]{pstricks}
\usepackage{epsfig}
\usepackage{pst-node}
\usepackage{pst-grad} 
\usepackage{pst-plot} 
\usepackage{pst-3dplot}

\newcommand{\N}{\mathbb{N}}

\newcommand{\Z}{\mathbb{Z}}

\newcommand{\mS}{\mathcal{S}}

\theoremstyle{plain}
\newtheorem{theorem}{Theorem}[section]
\newtheorem{lemma}[theorem]{Lemma}
\newtheorem{corollary}[theorem]{Corollary}
\newtheorem{proposition}[theorem]{Proposition}

\theoremstyle{definition}

\newtheorem{example}[theorem]{Example}

\newtheorem{remark}[theorem]{Remark}

\title[Numerical Semigroups II: Pseudo-Symmetric AA-Semigroups]{Numerical Semigroups II:\\Pseudo-Symmetric AA-Semigroups}
\author[I. Garc\'{i}a-Marco]{Ignacio Garc\'{i}a-Marco\,*}
\address{LIP, ENS Lyon - CNRS -UCBL - INRIA, Universit\'e de Lyon UMR 5668, Lyon, France}
\thanks{* Corresponding Author: Phone: +33-624564368. Email: ignacio.garcia-marco@ens-lyon.fr,\, iggarcia@ull.es \\ \indent 
The first author was supported by the Ministerio de
Econom\'ia y Competitividad,  Spain (MTM2013-40775-P)}
\email{ignacio.garcia-marco@ens-lyon.fr, iggarcia@ull.es}
\author[J.L. Ram\'irez Alfons\'in]{Jorge L. Ram\'irez Alfons\'in}
\address{Universit\'e de Montpellier, Institut Montpelli\'erain Alexander Grothendieck, Case Courrier 051, Place Eug\`ene Bataillon, 34095 Montpellier Cedex 05, France}
\email{jramrez@univ-montp2.fr}
\author[\O. J. R\o dseth]{\O ystein J. R\o dseth}
\address{Department of Mathematics, University of Bergen, Johs.
  Brunsgt. 12, N-5008 Bergen, Norway} 
\email{rodseth@math.uib.no}



\begin{document}

\begin{abstract}
This paper is a continuation of the paper ``Numerical Semigroups: Ap\'{e}ry Sets and Hilbert Series" \cite{j&o}. 
We consider the general numerical AA-semigroup, i.e., semigroups consisting of all non-negative integer linear combinations
of relatively prime positive integers of the form $a,a+d,a+2d,\dots,a+kd,c$. We first prove that, in contrast to arbitrary numerical semigroups, there exists an upper bound for the {\em type} of AA-semigroups that only depends on the number
of generators of the semigroup. We then present two characterizations of {\em pseudo-symmetric} AA-semigroups. The first one leads to a polynomial time algorithm to decide whether an AA-semigroup is pseudo-symmetric. The second one gives a method to construct pseudo-symmetric AA-semigroups and provides explicit families of pseudo-symmetric semigroups with arbitrarily large number of generators. 
\end{abstract}

\maketitle

\section{Introduction}\label{introduction}

This paper is a continuation of the paper \cite{j&o}. For a numerical semigroup $\mS$, we recall that 
the \textit{Frobenius number} $g=g(\mS)$ is the largest integer not in $\mS$, 
and the \textit{genus} $N=N(\mS)$ is the number of non-negative integers not in $\mS$. 
The semigroup $\mS$ is {\it symmetric} if 
\[
\mS\cup(g-\mS)=\Z,
\]   
where $g-\mS=\{g-s\mid s\in \mS\}$. The semigroup is \textit{pseudo-symmetric} if the Frobenius number $g$ is \textit{even} and 
\begin{equation} \label{gS}
\mS\cup(g-\mS)=\Z\setminus\{g/2\}.
\end{equation}
It is well known that $\mS$ is symmetric if and only if $g=2N-1$. Similarly, it is also known that $\mS$ is pseudo-symmetric if and only if
\begin{equation} \label{g2N}
g=2N-2.
\end{equation}
We include an easy proof of this result in Section \ref{sec:ap} (Lemma \ref{pslemma}).

Set $\Delta=2N-1-g$. Then $\mS$ is symmetric if and only if $\Delta=0$, and pseudo-symmetric if and only if $\Delta=1$. The numerical semigroup $\mS$ is \textit{irreducible} if it is 
not the intersection of two strictly larger numerical semigroups. It now follows from \cite{bdf, fgh,r&b} that $\mS$ is irreducible if and only if $\Delta\leq1$.  

For a semigroup $\mS$ we set $\mS' = \{x \notin \mS \, \vert \, x+s \in \mS$ for all $s \in \mS\}.$ The elements
of $\mS'$ are usually called {\it pseudo-Frobenius numbers} and the number of elements
of $\mS'$ is called the {\it type} of $\mS$ and denoted by ${\rm type}(\mS)$. We notice that $g$ is always a pseudo-Frobenius number. 
Moreover, by \cite[Proposition 2]{fgh}
$\mS$ is symmetric if and only if $\mS' = \{g\}$, or equivalently, if the type of $\mS$ is $1$. Also, $\mS$ is pseudo-symmetric if and only if
$\mS' = \{g, g/2\}$, which implies that every pseudo-symmetric semigroup has type $2$.

The {\it Ap\'ery set} of $\mS$ with respect to $m \in \mS$ is defined as
$${\rm Ap}(\mS;m) = \{s \in \mS\, \vert \, s-m \notin \mS\}.$$ 
An Ap\'ery set of a semigroup $\mS$ is very difficult to determine in general. This set contains many
relevant information about the semigroup. As we shall point out, in Section \ref{sec:ap}, all the above mentioned parameters related to $\mS$
can be expressed in terms of an Ap\'ery set.

In this paper we focus our attention on numerical AA-semigroup consisting of all non-negative integer linear combinations of relatively prime positive integers $a,a+d,a+2d,\dots,a+kd,c$, where also $a,d,k,c$ are positive integers. 
In \cite{rod2}, R\o dseth presented ``semi-explicit" formulas for ${\rm Ap}(\mS; a)$, $g(\mS)$ and $N(\mS)$ when $\mS$ is  an AA-semigroup. 
\smallskip

This self-contained paper is organized as follows. In next section, we review some basic results on Ap\'ery sets and the type of numerical semigroups. We also present  
a useful characterization on the pseudo-symmetry of semigroups (Lemma \ref{pslemma}).
\smallskip

In Section \ref{sec:aap}, after recalling some basic notions and results on AA-semigroup given in \cite{rod2} needed for the rest of the paper, we prove  that the type of an AA-semigroup is at most $2k$ (Theorem \ref{cmtype}). Since  every three generated numerical semigroup is an AA-semigroup (with $k=1$), then Theorem \ref{cmtype} generalizes a result due to  Fr\"{o}berg, Gottlieb and H\"{a}ggkvist \cite[Theorem 11]{fgh} stating that a 3 generated numerical semigroup has type at most $2$.
\smallskip

In Section \ref{sec:char}, we present two
characterizations for pseudo-symmetric AA-semigroups. The first one (Theorem \ref{pseudo}) provides a criterion 
to decide when an AA-semigroup is irreducible. As  a consequence of this result and \cite[Theorem 6]{j&o}, we obtain a polynomial time algorithm to decide whether an AA-semigroup is pseudo-symmetric. Also, by combining the characterization of symmetric AA-semigroups given in \cite[Theorem 5]{j&o} together with Theorem \ref{pseudo} we obtain a
complete characterization of irreducible AA-semigroups. 
The second characterization of pseudo-symmetric AA-semigroups  (Theorem \ref{construccion}) shows how to construct any pseudo-symmetric AA-semigroup. In particular, it provides explicit families of pseudo-symmetric semigroups with arbitrarily large number of generators. 
These results extend those given by Rosales and Garc\'ia-Sanchez in \cite{r&g} who characterized three generated pseudo-symmetric numerical semigroups.

\section{The Ap\'{e}ry set and the type of a numerical semigroup} \label{sec:ap}

Ap\'ery sets are in general very difficult to describe and calculate. Nevertheless, in the few cases where one knows explicitly this set, it
 provides a lot of interesting information about the numerical semigroup.  In this section we
recall some nice properties of the Ap\'ery set and use them to prove some new results, all these properties will be useful in the sequel.

Firstly, it is well known that the Frobenius number of $\mS$ can be computed as
\begin{equation} \label{frob-apery} g(\mS)=\max {\rm Ap}(\mS;m)-m, \end{equation}
and also 
\begin{equation} \label{formul-N}
N(\mS)=\frac{1}{m}\sum\limits_{w\in Ap(\mS;m)}w-\frac{1}{2}(m-1),
\end{equation}
(a result essentially due to Selmer \cite{sel}).

Let $a_1,\ldots,a_n$ be relatively prime positive integers, and let $e$ be a positive integer prime to $a_1$. Put $\mS^e=\langle a_1,ea_2,\ldots,ea_n\rangle$.
By  \cite[Section 2]{j&o}, we have that the map $\N \longrightarrow \N$; $x \mapsto ex$ induces a bijection  ${\rm Ap}(\mS;a_1) \longrightarrow {\rm Ap}(\mS^e; a_1)$ and, hence, 
\begin{equation}\label{apery} e {\rm Ap}(\mS;m) = {\rm Ap}(\mS^e; m).\end{equation}
From this equality, it was deduced in  \cite[(9) and (10)]{j&o} that 
\[
 g(\mS^e)=eg(\mS)+a_1(e-1),\qquad N(\mS^e)=eN(\mS)+\frac{1}{2}(a_1-1)(e-1),
\]
and so, 
\begin{equation}\label{deltasym}
 \Delta(\mS^e)=2eN(\mS)+ a_1(e-1)-1-eg(\mS)-a_1(e-1)=2eN(\mS)-eg(\mS)-e=e\Delta(\mS).
\end{equation}
Thus if $\mS^e$ is symmetric for some admissible $e$, then $\mS^e$ is symmetric for all admissible numbers $e$. Moreover, we have that $\mS$ is pseudo-symmetric if and only if $\Delta(\mS^e)=e$. In particular, if $\mS^e$ is pseudo-symmetric, then $e=1$. Therefore, if $\mS=\langle a,a+d,\ldots,a+kd,c\rangle$, it is no restriction to assume $\gcd(a,d)=1$ when
studying pseudo-symmetry.

One can also determine the set $\mS'$ of pseudo-Frobenius numbers in terms of the Ap\'ery set. Indeed, if we consider $\leq_{\mS}$ the partial order in $\Z$
given by $x \leq_{\mS} y \Longleftrightarrow y-x \in \mS$, then, by \cite[Proposition 7]{fgh}, $s \in \mS$ is a maximal element in ${\rm Ap}(\mS;m)$ with respect
to $\leq_{\mS}$ if and only if $s-m$ is a pseudo-Frobenius number of $\mS$.  From this result and (\ref{apery}) we directly deduce the following. 

\begin{proposition}\label{tipoconstante} ${\rm type}(\mS) = {\rm type}(\mS^e)$.
\end{proposition}

We now present a helpful characterization of pseudo-symmetric semigroups.

\begin{lemma} \label{pslemma}Let $\mS$ be a numerical semigroup. The following conditions are equivalent:
\begin{itemize}
\item[(a)] $\mS$ is pseudo-symmetric 
\item[(b)] If we write $w(i)$ for the unique $w\in Ap(\mS;m)$ satisfying $w\equiv i$ (mod $m$), we have that $$
w(g/2+i)+w(g/2-i)=w(g)+
\begin{cases}
m &\text{if $i\equiv0$ (mod $m$),}\\
0 &\text{otherwise.}
\end{cases}
$$
\item[(c)] $
g(\mS)=2N(\mS)-2.$
\end{itemize}
\end{lemma}

\begin{proof} 

[$(a) \Longrightarrow (b)$] For all $i$, it is clear that $w(g/2+i)+w(g/2-i) \equiv w(g) \ ({\rm mod} \ m)$, so 
$w(g/2+i)+w(g/2-i) - w(g) = \lambda_i m$ for some $\lambda_i \in \Z$. We observe that $w(g) = g(\mS) + m$, which 
implies that $\lambda_i \geq 0$, otherwise $g \in \mS$. Moreover, if $\lambda_i > 0$,
we get that $(w(g/2 - i) - m) + (w(g/2 + i) - \lambda_i m) = g(\mS)$. From this equality 
we derive that $w(g/2 - i) - m,\, w(g/2 + i) - \lambda_i m \notin \mS \cup (g - \mS)$. Since $\mS$ is pseudo-symmetric, we obtain that
$g/2 = w(g/2-i) - m = w(g/2 + i) - \lambda_i m$, so $\lambda_i = 1$ and $i \equiv 0 \ ({\rm \ mod}\ m)$.
\smallskip

[$(b) \Longrightarrow (a)$] Take $s \notin \mS \cup (g - \mS)$. Since $s \notin \mS$, there exists $i_0$ such that $w(g/2 - i_0) = s + \mu m$ with $\mu \in \Z^+$.
If $i_0 \not\equiv 0 \ ({\rm mod}\ m)$, then $s + \mu m + w(g/2 - i_0) = w(g) = g + m$ and we get that $s \in g - \mS$, a contradiction. If $i_0 \equiv 0\ ({\rm mod}\ m)$, then 
$2 s + 2 \mu m = 2 w(g/2) = w(g) + m = g + 2m$. Hence $\mu = 1$ and $s + \mu m = g/2 + m$, from where we finally obtain that $s = g/2$.
\smallskip

[$(b) \Longrightarrow (c)$] Notice that $g+m=w(g)$. Summing over a complete system of residues $i$ (mod $m$) and using formula (\ref{formul-N}), we get
\begin{align*}
2N-2&=\frac{1}{m}\sum_i (w(g/2+i)+w(g/2-i))-m-1\\
{}&= \frac{1}{m}\sum_i w(g)-m =g.
\end{align*}

[$(c) \Longrightarrow (b)$] We find that 
\begin{align*}
g&=2N-2=\frac{1}{m}\sum_{i=0}^{m-1} (w(g/2+i)+w(g/2-i))-m-1\\
{}&= \frac{1}{m}\sum_{i=0}^{m-1}(w(g)+c_im)-m-1=g+\sum_i c_i -1,
\end{align*}
for integers $c_i\geq0$. Thus, $c_i=0$ with the exception of one value of $i$ mod $m$ for which $c_i=1$. Since $c_i = c_j$ whenever $i + j = m$, we deduce that $c_0 = 1$ and thus
(b) holds.
\end{proof} 

\section{The type of almost arithmetic progressions} \label{sec:aap}

Let us quickly recall some notions and results given in \cite{rod2} needed in the rest of the paper.  Assume
that $\gcd(a,d)= 1$, let $s_{-1}=a$ and set $s_0$ the only integer such that
\[
 ds_0\equiv c\pmod{s_{-1}},\quad0\leq s_0<s_{-1}.
\]
If $s_0 = 0$ we set $m = -1$. Otherwise, we use the Euclidean algorithm with negative division remainders,
\begin{alignat*}{2}
s_{-1} & = q_1 s_0-s_1,         &&0\leq s_1<s_0;\\
s_0 & = q_2 s_1-s_2,         &&0\leq s_2<s_1;\\
s_1 & = q_3 s_2-s_3,         &&0\leq s_3<s_2;\\
   &\dots & \\
s_{m-2} & = q_m s_{m-1}-s_m,\quad &&0\leq s_m<s_{m-1};\\
s_{m-1} & = q_{m+1} s_m,     &&0=s_{m+1}<s_m.
\end{alignat*}

\noindent We have $s_m=\gcd(a,c)$. We define integers $P_i$ by $P_{-1}=0$, $P_0=1$, and (if $m\geq0$),
\[
 P_{i+1}=q_{i+1}P_i - P_{i-1},\quad i=0,\dots,m.
\]
Then, by induction on $i$,
\begin{equation*} \label{sP}
 s_iP_{i+1}-s_{i+1}P_i=a,\qquad i=-1,0\dots,m,
\end{equation*}
and
\[
 0 = P_{-1} < 1 = P_0<\dots<P_{m+1}=\frac{a}{s_m}.
\]
In addition we have,
\begin{equation} \label{Ps}
 ds_i\equiv cP_i\pmod{a},\quad i=-1,\ldots,m+1.
\end{equation}
Putting
\[
 R_i=\frac{1}{a}\left((a+kd)s_i-kcP_i\right),
\]
we see that all the $R_i$ are integers. Moreover, we have
$R_{-1}=a+kd$, $\displaystyle R_0=\frac{1}{a}\left((a+kd)s_0-kc\right)$, and
\[
 R_{i+1}=q_{i+1}R_i - R_{i-1},\quad i=0,\dots,m.
\]
Furthermore,
\[
 -\frac{c}{s_m}=R_{m+1}<R_m<\dots<R_0<R_{-1}=a+kd,
\]
so there is a unique integer $v$ such that
\[
 R_{v+1}\leq0<R_v.
\]

For each $i$, let $\tau(i)$ be the smallest integer for which there exist non-negative integers $x,y$ such that
\begin{equation} \label{tau}
\tau(i)=(a+kd)x+kcy,\quad dx+cy\equiv i\pmod{a}.
\end{equation}

If there are more than one such pair $x,y$ for some $i$, choose the one with $y$ minimal. This gives us a unique set $L$ of pairs of nonnegative integers, with $|L|=a$. The set $L$ consists of all lattice points in a closed L-shaped region in the $x,y$-plane. Sometimes the L-shape degenerates to a rectangle or an interval. R\o dseth \cite{rod2} proved that $L=A\cup B$, where
$$\begin{array}{llll} A & = & \{(x,y) \, \vert \,  0\le x\le s_v-1,& 0\le y\le P_{v+1}-P_{v}-1\} \\ B & = & \{(x,y) \, \vert \, 0\le x\le s_v-s_{v+1}-1, & P_{v+1}-P_v\le y\le P_{v+1}-1 \}. \end{array}$$
Moreover, R\o dseth showed that if we consider the map

\begin{align}\label{apRodseth}
& \varphi:  \ \ \Z^2  \ \ \rightarrow  \ \ \Z \\ \nonumber
& \ \ \ \ \ (x,y) \mapsto \   \left\lceil \frac{x}{k} \right\rceil a + xd + yc,\\
& \text{then } {\rm Ap}(\mS;a) = \varphi(L). \nonumber
\end{align}

Combining (\ref{apRodseth}) and (\ref{frob-apery}), R\o dseth easily deduced that

\begin{equation}\label{frob-num}
g(\mS)={\rm max}\{\varphi(s_v-1,P_{v+1}-P_v-1), \varphi(s_v-s_{v+1}-1,P_{v+1}-1)\} - a.
\end{equation}



Fr\"oberg, Gottlieb and H\"{a}ggkvist proved in \cite[Theorem 11]{fgh} that every three generated numerical semigroup has type at most $2$. In the same 
paper the authors include an observation due to J. Backelin showing that there is no upper bound on the type of $\langle a_1,\ldots,a_t \rangle$ only in terms of $t$ for $t \geq 4$.
Indeed, he provided a family $(\mS_n)_{n \in \N}$ of $4$ generated numerical semigroups whose type is $2n + 3$. In this section we consider AA-semigroups, i.e.,
numerical semigroups of the form $\mS = \langle a, a + d,\ldots,a+kd, c\rangle$, where $a,d,k,c$ are positive integers. These semigroups generalize three generated 
semigroups. We prove that the type of AA-semigroups is bounded in terms of the number of generators of the semigroup. 

\begin{theorem}\label{cmtype} Let $\mS$ be an AA-semigroup, then $type(\mS) \leq 2k$.
\end{theorem}

By Proposition \ref{tipoconstante}, it suffices to prove this result when $\gcd(a,d) = 1$. 

%


\medskip

\noindent {\bf Proof of Theorem \ref{cmtype}.} It suffices to observe that for all $(x,y) \in \Z^2$, $\varphi(x+k,y) - \varphi(x,y) = a + kd \in \mS$ and that $\varphi(x,y+1) - \varphi(x,y) = c \in \mS$. The rectangular grid shape of $A$ shows that in $\varphi(A)$ there are at most $k$ maximal elements of ${\rm Ap}(\mS;a)$ with respect of $\leq_{\mS}$. The same argument works for $B$. So, there are at most $2k$ maximal elements in ${\rm Ap}(\mS;a)$ with respect to $\leq_{\mS}$ and, hence, ${\rm type}(\mS) \leq 2k$. 
\hfill$\Box$

\medskip

The bound provided in Theorem \ref{cmtype} is sharp. Indeed, consider the following

\begin{example}For all $k \geq 1$ we consider the AA-semigroup $\mS_k = \langle a,\ldots,a+kd,c\rangle$ with $a := 3k+2, d := 1$ and $c := 5k+3$. We observe that
$s_0 = 2k+1$. The equality $a = 2 s_0 - k$ yields that $s_1 = k$ and $P_1 = 2$. Since $R_1 = -2k \leq 0 < k+1 = R_0$, we get that $v = 0$. A
direct application of (\ref{apRodseth}) gives that ${\rm Ap}(\mS;a) = \varphi(L)$, where $L$ is the set of Figure
\ref{fig:ejemplo}.

 \begin{figure}[htb]
  \centering
  \includegraphics{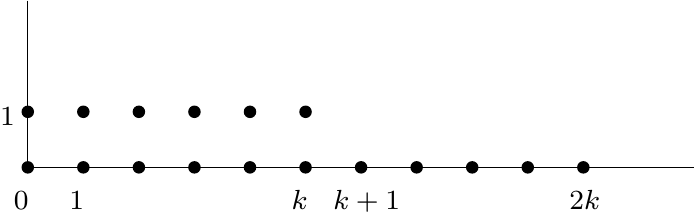}
  \caption{L-shape in bijection with ${\rm Ap}(\mS_k;a)$}
  \label{fig:ejemplo}
 \end{figure}

Then, 
 $${\rm Ap}(\mS;a) = \{0,a,a+1,\ldots,a+k,2a+k+1,\ldots,2a+2k,c,c+a+1,\ldots,c+a+k\}$$ Hence, the maximal elements of ${\rm Ap}(\mS;a)$ 
with respect to $\leq_{\mS}$ are $\{2a+k+i, c+ a + i \, \vert \, 1 \leq i \leq k\}$ (see Figure \ref{fig:ejemplo2}) and ${\rm type}(\mS) = 2k$.
  \begin{figure}[htb]
  \centering
  \includegraphics{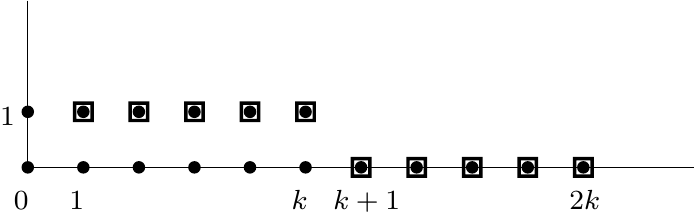}
  \caption{Points corresponding to maximal elements of ${\rm Ap}(\mS_k;a)$}
  \label{fig:ejemplo2}
 \end{figure}

\end{example}

\section{Pseudo-symmetric almost arithmetic semigroups} \label{sec:char}

The goal of this section is to characterize pseudo-symmetry for AA-semigroups. When $\mS$ is a three
generated numerical semigroup, Rosales and Garc\'ia-S\'anchez 
gave a characterization of pseudo-symmetry in \cite{r&g}. They proved that if $\mS$ is a three generated 
numerical semigroup then  $\mS$ is pseudo-symmetric if and only $\mS = \langle c(b-1) + 1, (c-1)a + 1,  b(a-1) + 1 \rangle$ with ${\rm gcd}(c(b-1) + 1, (c-1)a + 1) = 1$.
\smallskip

Here, we only consider the case when $\mS$ is an AA-semigroup generated by at least $4$ elements, i.e., when $k \geq 2$.
However, one could easily extend our proof for three generated semigroups and recover the result of Rosales and Garc\'ia-S\'anchez. 
\smallskip

As mentioned in Section \ref{sec:ap}, if $\gcd(a,d) \neq 1$, then $\mS$ is not pseudo-symmetric, so from now on we suppose that $a$ and $d$
are relatively prime. The idea of our proof is to apply the description of ${\rm Ap}(\mS;a)$ in terms of the set $L$ and the map $\varphi$ 
given in (\ref{apRodseth}) together with the characterization of pseudo-symmetry given in Lemma \ref{pslemma} (b). 
\smallskip

In the forthcoming we denote by $(x_0,y_0)$ the element of $L$ such that $\varphi(x_0,y_0) = g(\mS) + a$. 

By  (\ref{frob-num}) it follows that 
$$\hbox{$(x_0,y_0)=(s_v-s_{v+1}-1,P_{v+1}-1)$ or $(s_v-1,P_{v+1}-P_v-1)$.}$$

Before stating and proving the main results we need the following
three lemmas. The proof of the first one is a straightforward checking (we leave it to the reader).

\begin{lemma}\label{biyeccion} Let $\mS$ be an AA-semigroup. 

If $(x_0,y_0) = (s_v-s_{v+1}-1,P_{v+1}-1)$, then 
\begin{itemize}
\item[(a1)] $\varphi(x,y) + \varphi(x_0 - x, y_0 - y) \equiv \varphi(x_0,y_0)\ ({\rm mod}\ a)$ for  $0 \leq x \leq x_0$, $0, \leq y \leq y_0$.
\item[(a2)] $\varphi(x,y) + \varphi(x_0 + s_{v} - x, y_0 - P_{v} - y) \equiv \varphi(x_0,y_0)\ ({\rm mod}\ a)$ for $x_0 < x \leq s_{v}-1$, $0 \leq y \leq y_0 - P_v$.
\end{itemize}

If $(x_0,y_0) = (s_v-1,P_{v+1}-P_v-1)$, then 
\begin{itemize}
\item[(b1)] $\varphi(x,y) + \varphi(x_0 - x, y_0 - y) \equiv \varphi(x_0,y_0)\ ({\rm mod}\ a)$ for  $0 \leq x \leq x_0$, $0, \leq y \leq y_0$.
\item[(b2)] $\varphi(x,y) + \varphi(x_0 - s_{v+1} - x, y_0 + P_{v+1} - y) \equiv \varphi(x_0,y_0)\ ({\rm mod}\ a)$ for $0 \leq x \leq x_0 - s_{v+1}$, $y_0 < y \leq y_0 + P_{v+1}$.
\end{itemize}
\end{lemma}

\begin{remark}\label{remarkx0} Let $\mS$ be a pseudo-symmetric AA-semigroup and let us denote by $(x_1,y_1)$ the element in $L$ such that $\varphi(x_1,y_1) = (g(\mS)/2) + a.$ Then,
\begin{itemize}
\item[(a)] $2 \varphi(x_1,y_1) = g(\mS) + 2a = \varphi(x_0,y_0) + a$, and 
\item[(b)] if $2 \varphi(x_1,y_1) = \varphi(x',y') + \varphi(x'',y'')$ with $(x',y'), (x'',y'') \in L$, then $x_1 = x' = x''$ and $y_1 = y' = y''$ (this is a 
consequence of Lemma \ref{pslemma}(b)).
\end{itemize}
\end{remark}

\begin{lemma}\label{lemax1}Let $\mS$ be a pseudo-symmetric AA-semigroup. Then,
\begin{itemize} \item[(a)] either $y_1 = 0$ or $(x_1,y_1 + 1) \notin L$,
\item[(b)] either $x_1 < k$ or $(x_1+k,y_1) \notin L$, and
\item[(c)] either $x_1 = 0$, $(x_1+1,y_1) \notin L$ or $x_1 \equiv 1 \ ({\rm mod}\ k)$.
\end{itemize} 
\end{lemma}
\begin{proof}
If (a) does not hold, then $2 \varphi(x_1,y_1) = \varphi(x_1,y_1-1) + \varphi(x_1,y_1+1)$, but this contradicts Remark \ref{remarkx0}(b). Similarly, if  (b) does not hold, then $2 \varphi(x_1,y_1) = \varphi(x_1-k,y_1) + \varphi(x_1+k,y_1)$, a contradiction. Finally, if (c) does not hold, then  $\varphi(x_1,y_1) = \varphi(x_1-1,y_1) + d$ and
$\varphi(x_1,y_1) \leq \varphi(x_1+1,y_1) - d$. Hence, $\varphi(x_0,y_0) + a = 2 \varphi(x_1,y_1) \leq \varphi(x_1-1,y_1) + \varphi(x_1+1,y_1) = \varphi(x_0,y_0)$, which is again a contradiction.
\end{proof}

\begin{lemma}\label{lemax0}Let $\mS$ be a pseudo-symmetric AA-semigroup such that $(x_0,y_0) \neq (2,0)$. Then,
either $x_0 = 0$ or $x_0  \equiv 1\ ({\rm mod} \ k)$. In particular, this implies that  
$$
\varphi(x_0,y_0) = \varphi(x_0 - x, y_0 - y) + \varphi(x,y)
$$ for all $0 \leq x \leq x_0$, $0 \leq y \leq y_0$.
\end{lemma}
\begin{proof} Firstly we observe that $\mS$ cannot be a two generated semigroup, otherwise it would be symmetric.
By contradiction, assume that $x_0 > 0$ and that $x_0 \not\equiv 1\ ({\rm mod} \ k)$. Then,
$\varphi(x_0,y_0) = \varphi(x_0-1,y_0) + d$.  We claim that $d \notin \mS$. Indeed, if $d \in \mS$,
then $\mS = \langle a, c \rangle$, a contradiction. Therefore, $x_1 = x_0-1$ and $y_1 = y_0$. The equality $2 \varphi(x_1,y_1) = \varphi(x_0,y_0) + a$ yields $\varphi(x_1,y_1) = a+d$.
Hence we have that either $(x_1,y_1) = (1,0)$ or $(x_1,y_1) = (0,\lambda)$ for some $\lambda > 0$. Then, $(x_0,y_0) = (2,0)$ or $x_0 = 1$, a contradiction. The reader can easily verify the second part of the result.
\end{proof}

We may now proceed with our first characterization. 

\begin{theorem}\label{pseudo} Let $\mS = \langle a, a+d, \ldots, a+kd, c\rangle$ with $a,d,c \in \Z^+$ and $k \geq 2$.
Then, $\mS$ is pseudo-symmetric if and only if $\gcd(a,d) = 1$ and one of the following statements holds:
\begin{itemize}
\item[(a)] $\mS = \langle 3, 3+d, 3+2d \rangle$
\item[(b)] $a(c-1) = 2(c+d)$ with $c$ odd,
\item[(c)] $s_{v} \equiv 3 \, ({\rm mod}\ k)$, $s_{v+1} = 1$, $P_{v+1} = P_v  + 1$ and $R_v = 3$,
\item[(d)] $s_{v} \equiv 1 \, ({\rm mod}\ k)$, $s_{v+1} = 2k-1$, $P_{v+1} = P_v  + 1$ and $R_v = 1$,
\item[(e)] $s_0 \equiv 2\, ({\rm mod} \ k)$, $s_0 = s_1 + 1$ and $R_1 = 1 - 2k$, or
\item[(f)] $s_0 \equiv 2\, ({\rm mod} \ k)$, $s_0 = s_1 + 3$ and $R_1 = - 1$.

\end{itemize}
\end{theorem}
\begin{proof}
$(\Rightarrow)$  As we mentioned in Section 2, $\gcd(a,d) = 1$ is a necessary condition for $\mS$ to be pseudo-symmetric.

We first consider $(x_0,y_0) = (2,0)$. We observe that $(x_1,y_1) = (1,0)$. Otherwise $\varphi(x_0,y_0) - \varphi(1,0) = d \in \mS$ by Lemma \ref{pslemma}, and then
$\mS = \langle a, c\rangle$, which is symmetric, a contradiction. If $a = 3$, then ${\rm Ap}(\mS; 3) = \{0,3+d,3+2d\}$ and $\mS = \langle 3,3+d,3+2d\rangle$, getting (a). Assume now
that $a > 3$. We have that $(1,1) \notin L$, otherwise $\varphi(x_0,y_0) - \varphi(1,1) = d - c \in \mS$, but this implies that $\mS = \langle a, c\rangle$, a contradiction. Then, 
the L-shape is $L = \{(1,0),(2,0)\} \cup \{(0,\lambda)\, \vert \, 0 \leq \lambda \leq P_{v+1}-1\}$. From here we deduce that $s_v = 3$, $s_{v+1}=2$, $a = P_{v} + 3 = P_{v+1}+2$. Moreover, the equality $\varphi(0,i) + \varphi(0,P_{v+1}-i) = \varphi(2,0)$ for all $i \in \{1,\ldots,P_{v}\}$ yields that $(a-2)c = P_{v+1} c = a + 2d$ or, equivalently, that 
$a(c - 1) = 2 (c + d)$. To get (b) it only remains to prove that $c$ is odd. We proceed by contradiction. Assume that $c$ is even, then the equality $a(c - 1) = 2 (c + d)$ proves that $a$ is also even. 
Since $g(\mS) = 2d$, then $\mS = \langle a, a+d,c \rangle$. Hence, if we set  $\mS' = \langle a/2, a + d, c/2 \rangle$, then by (\ref{deltasym}) we get that $\Delta(\mS) = 2\Delta(\mS') \neq 1$,  and $\mS$ is not pseudo-symmetric.

From now on, we assume that $(x_0,y_0) \neq (2,0)$, then by Lemma \ref{lemax0} it follows that either $x_0 = 0$ or $x_0 \equiv 1 \, ({\rm mod}\ k)$.
We split the proof in two cases according to the value of $g(\mS)$ given by equation (\ref{frob-num}).
\smallskip

{\emph Case I:} $g(\mS) = \varphi(s_v-s_{v+1}-1,P_{v+1}-1) - a.$
If $s_{v+1} = 0$,  then we observe that the L-shape degenerates to a rectangle and, by the second part of Lemma \ref{lemax0}, we obtain that $\mS$ is symmetric  (we could also have obtained that $\mS$ is symmetric by applying \cite[Theorem 5]{j&o}). If $s_{v+1} > 0$, then  by Lemma \ref{biyeccion} we have that $x_1 = (2s_v - s_{v+1} - 1)/2$ and $y_1 = (P_{v+1} - P_v - 1)/2$. Moreover, by Lemma \ref{lemax1} (a) we have that $y_1 = 0$ and, hence, $P_{v+1} - P_v = 1$.  Moreover, by Lemma \ref{lemax1} (b) we also have that $(x_1 + k,y_1) \notin L$ and, hence, $x_1 + k \geq s_v$, which implies that $s_{v+1} \leq 2k-1$.

Since $x_1 > 0$ then, by Lemma \ref{lemax1} (c), we have two subcases.
\smallskip

\emph{Subcase 1:} If $(x_1 + 1,0)\notin L$ or, equivalently, $s_{v+1} = 1$. According to Lemma \ref{lemax0} we have that either $x_0 = 0$ or $x_0 \equiv 1 \ ({\rm mod}\ k)$. If $x_0 = 0$, then $x_1 = 1$ and the identity $2 \varphi(x_1,y_1) = \varphi(x_0,y_0) + a$ implies that $P_v c = a + 2d$. However, this is not possible because $0 < R_v = \frac{1}{a}
[(a+kd)s_v - kcP_v] = 2-k \leq 0$. If $x_0 \equiv 1 \ ({\rm mod}\ k)$. Since $k \geq 2$, we observe that $\lceil x_0/k \rceil = \lceil (x_0+1)/k \rceil = (x_0 - 1 + k)/k$ . Thus the identity  $2 \varphi(x_1,y_1) = \varphi(x_0,y_0) + a$ yields \begin{equation}\label{caso1}P_v c = \frac{x_0-1}{k}a + (x_0+2)d=\frac{s_v-3}{k}a + ds_v.\end{equation} From here we directly deduce that $R_v = \frac{1}{a}[(a+kd)s_v - k c P_v] = 3$ and we obtain (c).
\smallskip

\emph{Subcase 2:} If $x_1 \equiv 1 \ ({\rm mod}\ k)$.  According to Lemma \ref{lemax0} we have that either $x_0 = 0$ or $x_0 \equiv 1 \ ({\rm mod}\ k)$. If $x_0 = 0$, then $x_1 = \lambda k+1$ and $(x_1+k,0) \notin L$ and thus $\lambda = 0$. This implies that the Ap\'ery set has $P_v + 2 = a$ elements and that $P_v c = 2(a+d) - a$ and thus $(a-2)c = a + 2d$. However, the fact that $R_v > 0$ implies that $(a+kd)s_v - kcP_v > 0$, but this can only happen if $2a + 2kd - kc(a-2) = 2a + 2kd - k(a+2d) > 0$, but this implies $k = 1$, a contradiction. If $x_0 \equiv 1 \ ({\rm mod}\ k)$, then we get that $x_0 = s_v - 2k$ and $x_1 = s_v - k$. Moreover, $P_v c = a (s_v-1)/k + d s_v$ and, hence, $kcP_v = (a+kd)s_v - a$. This implies that $R_v = 1$ and we obtain (d).
\smallskip

Case II: $g(\mS) = \varphi(s_v-1,P_{v+1}-P_v-1) - a.$
Firstly, we shall prove that 
\begin{itemize} 
\item[(i)] $v = 0$
\item[(ii)] $s_0 \equiv 2\ ({\rm mod}\ k)$
\item[(iii)] $s_0 - s_{1} \in \{1,3\}$
\end{itemize}

Since $x_0 \not= 0$, by Lemma \ref{lemax0}, we have that $x_0 \equiv 1\ ({\rm mod}\ k)$ 
and that $g(\mS) + a = \varphi(x,y) + \varphi(x_0-x,y_0-y)$ for all $0 \leq x \leq x_0$ and $0 \leq y \leq y_0$. Then, by Lemma \ref{biyeccion}, 
$x_1 =  (x_0 - s_{v+1})/2$, $y_1 =  (y_0 + P_{v+1})/ 2$ and by Lemma \ref{lemax1} we get also that $y_1 = P_{v+1}-1$ and that $P_v = 1$ yielding to $v = 0$. 
\smallskip

We observe that (d) is equivalent to prove that $x_1 = 0$ or $x_1 = 1$. By Lemma \ref{lemax1} we have that
$x_1 = 0$  or $x_1 \equiv 1\ ({\rm mod}\ k)$. Nevertheless, also by Lemma \ref{lemax1} we obtain that $x_1 < k$, so we conclude that $x_1 \in \{0,1\}$.

Firstly we assume that $s_0 - s_1 = 1$, then we have that  $a = q_1 s_0 - s_1 = (q_1-1)s_0 + 1$.

Since $v = 0$, we have that $R_1 = q_1 R_0 - (a + kd) \leq 0$ and, hence, $a + kd = q_1 R_0 + \lambda$ for some $\lambda \geq 0$.
Moreover, $a R_0 = (a+kd)s_0  - kcP_0 = (q_1 R_0 + \lambda)s_0 - kc$, so we deduce that $kc = (R_0+\lambda)(s_0-1)+ \lambda$.
We observe that $x_1 = 0$, $y_1 = y_0 + 1$. Then the equality $2\varphi(x_1,y_1) = \varphi(x_0,y_0) + a$ yields $q_1 c = (\frac{s_0 - 2}{k} + 2)a + (s_0 - 1)d$, which implies that $kq_1c = (s_0-1)(a+kd)+(2k-1)a$. Then $q_1 (\lambda s_0 + s_0 R_0 - R_0) = (s_0 - 1)(q_1 R_0 + \lambda) + (2k-1)a$ and we derive that $\lambda a = (2k-1)a$. Therefore, $\lambda = 2k-1$ and thus obtaining (e).
\smallskip

For $s_0 - s_1 = 3$, then we have that  $a = q_1 s_0 - s_1 = (q_1-1)s_0 + 3$.

Since $v = 0$, we have that $R_1 = q_1 R_0 - (a + kd) \leq 0$ and, hence, $a + kd = q_1 R_0 + \lambda$ for some $\lambda \geq 0$.
Moreover, $a R_0 = (a+kd)s_0  - kcP_0 = (q_1 R_0 + \lambda)s_0 - kc$, so we deduce that $kc = (R_0+\lambda)(s_0-3)+ 3\lambda$.
We observe that $x_1 = 1$, $y_1 = y_0 + 1$. Then the equality $2\varphi(x_1,y_1) = \varphi(x_0,y_0) + a$ yields $q_1 c = (\frac{s_0 - 2}{k})a + (s_0 - 3)d$, which implies that $kq_1c = (s_0-3)(a+kd)+a$. Then $q_1 (\lambda s_0 + s_0 R_0 - 3R_0) = (s_0 - 3)(q_1 R_0 + \lambda) + a$ and we derive that $\lambda a = a$. Thus, $\lambda = 1$ and we conclude (f).

$(\Leftarrow)$ In (a) it is clear that ${\rm Ap}(\mS,3) = \{0,3+d,3+2d\}$ and that $\mS$ is pseudo-symmetric. In (b), the equality $\frac{c-1}{2} a = c+d $
yields that $s_0 = a-1$, then $s_i = a-1-i, P_i = i+1$ and $q_i = 2$ for all $i$. We also observe that $R_{a-4} > 0$, indeed, $$\begin{array}{lll} R_{a-4}   
& =  (3 (a+kd)  - kc(a-3))/a  & = 3 + (3d - (a-3)c)k/a \\ & = 3 + (3d - a - 2d + c)k/a \\ & = 3-k + k(d+c)/a \\ & = 3 + k(c+1)/2 > 0,\end{array}$$ 
and that $R_{a-3} = (2(a+kd) - kc(a-2))/a = (2(a+kd) - k(a+2d))/a = 2-k \leq 0$. Then $v = a-4$, the L-shape corresponds to the one in Figure \ref{fig:Ldemo} and it is straightforward
to check that $\mS$ is pseudo-symmetric.

 \begin{figure}[htb]
  \centering
  \includegraphics[scale=.5]{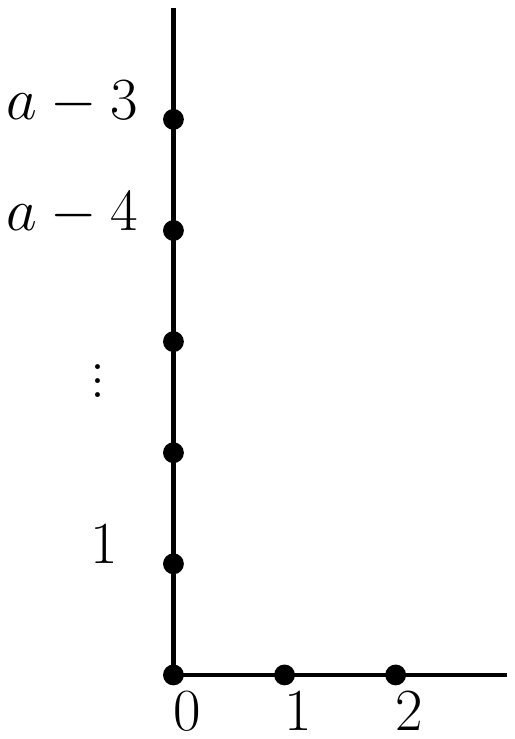}
  \caption{L-shape in Theorem \ref{pseudo}(b)}
  \label{fig:Ldemo}
 \end{figure}

 Both in (c) and (d) we have that $(x_0,y_0) = (s_v-s_{v+1}-1,P_{v+1}-1)$ and it is easy to check
that $\mS$ satisfies Lemma \ref{pslemma} (b) and, hence, it is pseudo-symmetric. For (e) and (f)
we observe that $v = 0$, we obtain that $(x_0,y_0) = (s_0-1,P_{1}-2)$ and, again, it is easy to check
that $\mS$ satisfies Lemma \ref{pslemma} (b) and, hence, $\mS$ is pseudo-symmetric.
\end{proof}

\bigskip
From Theorem \ref{pseudo} we deduce that the decision problem that receives as input $a,k,d,c$ and asks if the corresponding AA-semigroup $\mS$ is pseudo-symmetric is solvable in polynomial time. Indeed, Theorem \ref{pseudo} yields an algorithm for determining whether an AA-semigroup is pseudo-symmetric that relies on the computation
of some values of $s_i, R_i, P_i$ and, in \cite[Theorem 6]{j&o}, the authors proposed a polynomial time method to compute these values. As we mentioned in the introduction, a semigroup is irreducible if and only if it is either symmetric or pseudo-symmetric. In \cite{j&o}, it is proved that one can
check in polynomial time whether an AA-semigroup $\mS$ is symmetric. Hence, we directly derive the following.

\medskip

\begin{corollary}Given $a,k,d,c \in \Z^+$, the decision problem: "is the AA-semigroup $\mS = \langle a, a+d, \ldots, a+kd,c\rangle$ irreducible?", is solvable in polynomial time.
\end{corollary}

\smallskip

Our second characterization gives a method to construct any pseudo-symmetric AA-semigroup. We first need the following
 
\begin{lemma} \label{lemmaP}
If $P_{j+1}-P_j=1$ for some $j$ in the interval $-1\leq j\leq m$, then 
\[
q_{i+1}=2\quad\text{for}\quad i=0,1,\ldots, j,
\]
and
\[
 P_i=i+1\quad\text{for}\quad i=-1,0\ldots,j+1.
\]
\end{lemma}

\begin{proof} 
We have 
\[
	P_{i+1}=q_{i+1}P_i-P_{i-1}\geq2P_i-P_i,
\]
so that $P_{i+1}-P_i\geq P_i-P_{i-1}$ for all $i$. Hence
\[
1=P_{j+1}-P_j\geq P_j-P_{j-1}\geq\cdots\geq P_0-P_{-1}=1,
\]
and the lemma follows.
\end{proof}

\begin{remark}\label{remarkfacil} The condition $s_{v+1} = 1$ is equivalent to $v + 1 = m$ and $s_m = \gcd(a,c) = 1$.
Then, by Lemma \ref{lemmaP}, Theorem \ref{pseudo} (a) is equivalent to
$\gcd(a,c)=1$, $s_{m-1} \equiv 3 \, ({\rm mod}\ k)$, $P_{m+1} = m+2$ and $R_{m-1} = 3$.
\end{remark}

\begin{theorem}\label{construccion} Let $\mS = \langle a, a+d, \ldots, a+kd, c\rangle$ with $a,d,c \in \Z^+$ and $k \geq 2$.
Then, $\mS$ is pseudo-symmetric if and only if $\gcd(a,d) = 1$ and one of the following conditions holds:
\begin{itemize}
\item[(a)] $\mS = \langle 3, 3+d, 3+2d \rangle$
\item[(b)] $a(c-1) = 2(c+d)$ with $c$ odd,

\item[(c.1)] $a \equiv 1 {\text \ or \ }5\ ({\rm mod}\ 6), c(a-3) = 6d,{\text\ and \ }a \leq 2kd + 3,$

\item[(c.2)] there exist $\lambda,\mu \geq 1$ and $m \geq 1$ such that
$$a = (2+\lambda k)(m+1) + 1,\, c + (2+\lambda k)d = \mu a,$$
$${\text\ and \ } a+kd = \frac{cm - 3d}{\lambda}, $$ 

\item[(d)] there exist $\lambda \geq 2$, $\mu \geq 1$ and $v \geq 0$ such that
$$a = (v + 1)(2+(\lambda-2)k) + (\lambda k + 1),\, c + (2+(\lambda-2) k)d = \mu a,$$
$${\text\ and \ } a+kd = \frac{c(v+1) - d}{\lambda}, $$ 

\item[(e)] there exist $q_1 \geq 2$, $R_0 \equiv 2\, ({\rm mod}\, k)$ and $s_0 \equiv 2 \ ({\rm mod}\, k)$ such that
$$a = (q_1 - 1)s_0 + 1,\, a+kd = q_1 R_0 + 2k - 1\, {\text \ and \ } kc = (R_0 + 2k - 1)(s_0 - 1) + 2k - 1,$$

\item[(f)] there exist $q_1 \geq 2$, $R_0 \equiv 2\, ({\rm mod}\, k)$ and $s_0 \equiv 2 \ ({\rm mod}\, k)$, $s_0 \neq 2$ such that
$$a = (q_1 - 1)s_0 + 3,\, a+kd = q_1 R_0 + 1\, {\text \ and \ } kc = (R_0 + 1)(s_0 - 3) + 3.$$

\end{itemize}
\end{theorem}

\begin{proof}Conditions (a) and (b) are equal to Theorem \ref{pseudo}(a) and (b), respectively.
We first start proving that the condition in Theorem \ref{pseudo} (c) is equivalent to either (c.1) or (c.2).
In Theorem \ref{pseudo} (c) we have $R_v = 3$ or, equivalently, $3a = (a+kd)s_v - kcP_v$, $s_{v+1} = 1$ and $P_{v+1} = P_v + 1$, so from the inequality $R_{v+1} \leq 0$ we obtain that
\begin{equation}\label{in1} (a+kd)s_{v+1} - kcP_{v+1} = a+kd
-kc -kcP_v = a - 2kd - kc - (s_v-3)(a+kd) \leq 0.\end{equation} Since $s_v > s_{v+1} = 1$ and $s_v \equiv 3 \ ({\rm mod}\ k)$, then $(\ref{in1})$ holds if and only if $s_v > 3$, 
or $s_v = 3$ and $a - 2kd - kc \leq 0$. 

Assume first that $s_v = 3$, then $a = s_v P_{v+1} - s_{v+1} P_v = 2 P_v + 3$. Moreover, $R_v = 3$ implies that $3d = cP_v$, 
so we get that $6d = c(a-3)$. We claim that $a$ is not multiple of $3$. Otherwise, since $a$ is odd, then $a-3$ is multiple of $6$ and we get 
that $c$ divides $d$ implying that $\mS = \langle a, c\rangle$, which is symmetric, a contradiction. 
Inequality (\ref{in1}) and the equality $6d = c(a-3)$ imply now that $a \leq 2kd + 3$, getting (c.1).

When $s_v > 3$, then $s_v = 3+\lambda k$ for some $\lambda \geq 1$. By Remark \ref{remarkfacil}, we get that $v = m-1$ and that $P_{m} = P_{m-1} + 1$.
By Lemma \ref{lemmaP} we get that $q_i = 2$ for all $i$ and, since, $s_m = 1$, $s_{m-1} = 3 + \lambda k$, we get that $s_{m-i} = 1 + i(2+\lambda k)$ for all
$i$. In particular, $a = s_{-1} = (2 + \lambda k) (m+1) + 1$ and $s_0= a - (2 + \lambda k)$. From the condition $d s_0 \equiv c \, ({\rm mod}\, a)$ 
we deduce that there exists $\mu \geq 1$ such that $\mu a = c + d(2+\lambda k)$. Finally, $R_{m-1} = 3$ implies 
that $a+kd = (cm - 3d)/\lambda$, getting (c.2).
\smallskip

Conversely, if (c.1) holds, then the equality $6d = c(a-3)$ yields that $6d \equiv -3c \ ({\rm mod}\ a)$ and $a$ is not a multiple of $3$,
so $d (a-2) \equiv c \ ({\rm mod}\, a)$ and we get that $s_0 = a-2$. From here we easily derive that $q_i = 2$, $s_i = a - 2(i+1)$ and $P_i = i+1$ for all $i$ and
$m = (a-3)/2$. From these values we easily get that $R_{m-1} = 3$ and by Remark \ref{remarkfacil} we obtain that we are under the conditions of Theorem \ref{pseudo} (c).
If (c.2) holds, then $s_0$ is the only integer in $0 \leq s_0 \leq a-1$ such that $d s_0 \equiv c \ ({\rm mod}\ a)$, which is $s_0 = a - (2+\lambda k) = (2+\lambda k)m+1$. Thus, the equalities $s_i = q_{i+2} s_{i+1} - s_{i+2}$ with $0 \leq s_{i+2} < s_{i+1}$ for all $i \leq m-1$ yield that $s_i = (2 + \lambda k)(m-i) + 1$ for all $i \in \{-1,\ldots,m\}$  and $q_i = 2$ for all
$i$. Hence, $P_{m+1} = m + 2$ and $s_{m-1} = 3 + \lambda k$. We also easily get that $R_{m-1} = 3$, so by Remark \ref{remarkfacil} we also recover the conditions of Theorem \ref{pseudo} (c).
\smallskip

The proof that (d) is equivalent to Theorem \ref{pseudo} (d) is analogue (but easier) to the previous one, with the only differences
that the role of $m$ is played by $v+1$ and that $R_v = 1$ directly implies that $R_{v+1} \leq 0$, so we do not need to impose the condition $R_{v+1} \leq 0$.
\smallskip

Let us prove that conditions in Theorem \ref{pseudo} (e) are equivalent to (e). Firstly, the fact that $s_1 = s_0 - 1$ implies that $a = (q_1 - 1) s_0 + 1$.
We observe that $P_1 = q_1$ and the condition $R_1 = 1-2k$ implies that $R_1 = q_1 R_0 - R_{-1} = q_1 R_0 - (a+kd) = 1-2k$.
Finally the equality $a R_0 = (a+kd)s_0 - kc$ yields that $kc = (R_0 + 2k - 1)(s_0 - 1) + 2k - 1$. If we take residues modulo $k$ in this equality we
conclude that $R_0 \equiv 2 \ ({\rm mod}\ k)$. The converse is straightforward. 
\smallskip

The proof that (f) is equivalent to Theorem \ref{pseudo} (f) is analogue to one above.
\end{proof}

\bibliographystyle{plain}

\end{document}